\documentclass[11pt]{amsart}
\usepackage[margin=37.5mm]{geometry}
\usepackage{amsmath,amssymb}
\usepackage{amsthm}

\newtheorem{thm}{Theorem}[section]
\newtheorem{prop}{Proposition}[section]
\newtheorem{lem}{Lemma}[section]
\newtheorem{cor}{Corollary}[section]
\newtheorem{defi}{Definition}[section]

\newtheorem{rem}{Remark}[section]

\begin{document}

\title{A type of shadowing and distributional chaos}
\author{Noriaki Kawaguchi$^\ast$}
\thanks{$^\ast$JSPS Research Fellow}
\subjclass[2010]{74H65; 37C50}
\keywords{shadowing; s-limit shadowing; chain transitive; distributional chaos; Mycielski set}
\address{Faculty of Science and Technology, Keio University, 3-14-1 Hiyoshi, Kohoku-ku, Yokohama, Kanagawa 223-8522, Japan}
\email{gknoriaki@gmail.com}

\maketitle

\markboth{NORIAKI KAWAGUCHI}{A TYPE OF SHADOWING AND DISTRIBUTIONAL CHAOS}

\begin{abstract}
For any continuous self-map of a compact metric space, we prove a saturation of distributionally scrambled Mycielski sets under a type of shadowing and the chain transitivity.
\end{abstract}

\section{introduction}

Shadowing is an important subject of topological studies of dynamical systems, and its implications for chaos have been examined in many papers so far (see, e.g. \cite{AC,OW}). Recently, Li et al. \cite{LLT} proved that Devaney chaos with shadowing implies the distributional chaos of type 1 (DC1), and more strongly, the existence of a distributionally scrambled Mycielski set. In this paper, by exploiting a method in \cite{LLT}, we prove a saturation of the distributionally scrambled Mycielski sets under a type of shadowing and the chain transitivity. It extends the result of \cite{LLT} in certain respects.

Throughout, $X$ denotes a compact metric space endowed with a metric $d$. We recall the definition of distributional $n$-chaos \cite{LO,TF}.

\begin{defi}
\normalfont
For a continuous map $f\colon X\to X$, an $n$-tuple $(x_1,x_2,\dots,x_n)\in X^n$, $n\ge2$, is said to be {\em distributionally $n$-$\delta$-scrambled} for $\delta>0$ if
\begin{equation*}
\limsup_{m\to\infty}\frac{1}{m}|\{0\le k\le m-1\colon\max_{1\le i<j\le n}d(f^k(x_i),f^k(x_j))<\epsilon\}|=1
\end{equation*}
for all $\epsilon>0$, and
\begin{equation*}
\limsup_{m\to\infty}\frac{1}{m}|\{0\le k\le m-1\colon\min_{1\le i<j\le n}d(f^k(x_i),f^k(x_j))>\delta\}|=1.
\end{equation*}
Let ${\rm DC1}_n^\delta(X,f)$ denote the set of distributionally $n$-$\delta$-scrambled $n$-tuples and let ${\rm DC1}_n(X,f)=\bigcup_{\delta>0}{\rm DC1}_n^\delta(X,f)$.
A set $S\subset X$ is said to be {\em distributionally $n$-scrambled} (resp. {\em $n$-$\delta$-scrambled}) if $(x_1,x_2,\dots ,x_n)\in{\rm DC1}_n(X,f)$ (resp. ${\rm DC1}_n^\delta(X,f)$) for any distinct $x_1,x_2,\dots,x_n\in S$. We say that $f$ exhibits the {\em distributional $n$-chaos of type 1} if there is an uncountable distributionally $n$-scrambled subset of $X$.
\end{defi}

Our results are based on a relation defined by Richeson and Wiseman for chain transitive maps \cite{RW}. We recall it in a form suitable for our use. First, we define the chain transitivity.

\begin{defi}
\normalfont
Let $f\colon X\to X$ be a continuous map. For $\delta>0$, a sequence $(x_i)_{i=0}^k$ of points in $X$, where $k$ is a positive integer, is called a {\em $\delta$-chain} of $f$ if $d(f(x_i),x_{i+1})\le\delta$ for every $0\le i\le k-1$. $f$ is said to be {\em chain transitive} if for any $x,y\in X$ and $\delta>0$, there is a $\delta$-chain $(x_i)_{i=0}^k$ of $f$ with $x_0=x$ and $x_k=y$.
\end{defi}

Let $f\colon X\to X$ be a chain transitive map. A $\delta$-chain $(x_i)_{i=0}^k$ of $f$ is said to be a {\em $\delta$-cycle} of $f$ if $x_0=x_k$, and $k$ is called the {\em length} of the $\delta$-cycle. Let $m=m(\delta)>0$ be the greatest common divisor of the lengths of $\delta$-cycles of $f$. A relation $\sim_\delta$ in $X^2$ is defined by for any $x,y\in X$, $x\sim_\delta y$ iff there is a $\delta$-chain $(x_i)_{i=0}^k$ of $f$ with $x_0=x$, $x_k=y$, and $m|k$. 

\begin{rem}
\normalfont
The following properties hold:
\begin{itemize}
\item[(P1)] $\sim_\delta$ is an open and closed $(f\times f)$-invariant equivalence relation in $X^2$,
\item[(P2)] Any $x,y\in X$ with $d(x,y)\le\delta$ satisfies $x\sim_\delta y$, so for any $\delta$-chain $(x_i)_{i=0}^k$ of $f$, we have $f(x_i)\sim_\delta x_{i+1}$ for each $0\le i\le k-1$, implying $x_i\sim_\delta f^i(x_0)$ for every $0\le i\le k$,
\item[(P3)] For any $x\in X$ and $n\ge0$, $x\sim_\delta f^{mn}(x)$,
\item[(P4)] There exists $N>0$ such that for any $x,y\in X$ with $x\sim_\delta y$ and $n\ge N$, there is a $\delta$-chain $(x_i)_{i=0}^k$ of $f$ with $x_0=x$, $x_k=y$, and $k=mn$.
\end{itemize}
\end{rem}

Fix $x\in X$ and let $D_i$, $i\ge0$, denote the equivalence class of $\sim_\delta$ including $f^i(x)$. Then, $D_m=D_0$, and  $X=\bigsqcup_{i=0}^{m-1}D_i$ gives the partition of $X$ into the equivalence classes of $\sim_\delta$. Note that every $D_i$, $0\le i\le m-1$, is clopen in $X$ and satisfies $f(D_i)=D_{i+1}$. We call $\mathcal{D}^\delta=\{D_i\colon 0\le i\le m-1\}$ the {\em $\delta$-cyclic decomposition} of $X$.

\begin{defi}
\normalfont
The relation $\sim$ in $X^2$ is defined by for any $x,y\in X$, $x\sim y$ iff $x\sim_\delta y$ for every $\delta>0$. It is a closed $(f\times f)$-invariant equivalence relation in $X^2$.
\end{defi}

\begin{rem}
\normalfont
The relations $\sim_\delta$, $\delta>0$, and $\sim$ in $X^2$ have the following properties.
\begin{itemize}
\item[(1)] Given any $\delta_1>\delta_2>0$, $x\sim_{\delta_2} y$ implies $x\sim_{\delta_1} y$ for all $x,y\in X$, so we have $\mathcal{D}^{\delta_1}\prec\mathcal{D}^{\delta_2}$, i.e., for any $A\in\mathcal{D}^{\delta_2}$, there is $B\in\mathcal{D}^{\delta_1}$ such that $A\subset B$.
\item[(2)] We say that a pair of points $(x,y)\in X^2$ is {\em chain proximal} if for any $\delta>0$, there is a pair of $\delta$-chains $((x_i)_{i=0}^k,(y_i)_{i=0}^k)$ of $f$ such that $(x_0,y_0)=(x,y)$ and $x_k=y_k$. It is easy to see that any chain proximal pair $(x,y)\in X^2$ for $f$ satisfies $x\sim_\delta y$ for every $\delta>0$, i.e., $x\sim y$.
\end{itemize}
\end{rem}

As an assumption of the main theorem, we introduce a type of shadowing with respect to the relation $\sim$.

\begin{defi}
\normalfont
Let $f\colon X\to X$ be a chain transitive map. Denote by $\mathcal{D}$ the set of equivalence classes of $\sim$.
\begin{itemize}
\item For $\delta>0$, a sequence $(x_i)_{i\ge0}$ of points in $X$ is called a {\em $(\mathcal{D},\delta)$-pseudo orbit} of $f$ if $f(x_i)\sim x_{i+1}$ and $d(f(x_i),x_{i+1})\le\delta$ for all $i\ge0$.
\item $f$ has the {\em shadowing property along  $\mathcal{D}$} if for any $\epsilon>0$, there is $\delta>0$ such that every $(\mathcal{D},\delta)$-pseudo orbit $(x_i)_{i\ge0}$ of $f$ is {\em $\epsilon$-shadowed} by some $x\in X$ with $x_0\sim x$, i.e., $d(f^i(x),x_i)\le\epsilon$ for all $i\ge0$.
\end{itemize}
\end{defi}

\begin{rem}
\normalfont
In general, by replacing $\mathcal{D}$ with any partition $\mathcal{P}$ of $X$, we can define $(\mathcal{P},\delta)$-pseudo orbits and shadowing property along $\mathcal{P}$. Note that when $\mathcal{P}=\{X\}$, they are the ordinary $\delta$-pseudo orbits and shadowing property.
\end{rem}

Given any $x\in X$, define $\mathcal{D}(x)\in\mathcal{D}$ by $x\in\mathcal{D}(x)$. Note that every $\mathcal{D}(x)$, $x\in X$, is a closed subset of $X$. Denote by $K(X)$ the set of non-empty closed subsets of $X$ endowed with the Hausdorff metric $d_H\colon K(X)\times K(X)\to[0,\infty)$; for any $A,B\in K(X)$,
\[
d_H(A,B)=\inf\{\epsilon>0\colon A\subset B_\epsilon(B), B\subset B_\epsilon(A)\},
\]
where $B_\epsilon(\cdot)$ denotes the $\epsilon$-neighborhood.

The main theorem of this paper is the following.

\begin{thm}
Suppose that a continuous map $f\colon X\to X$ satisfies the following properties:
\begin{itemize}
\item[(1)] $f$ is chain transitive,
\item[(2)] $f$ has the shadowing property along $\mathcal{D}$,
\item[(3)] $\mathcal{D}(\cdot)\colon(X,d)\to(K(X),d_H)$ is continuous,
\item[(4)] $h_{top}(f)>0$, where $h_{top}(\cdot)$ denotes the topological entropy.
\end{itemize}
Then, the following property holds:
\begin{itemize}
\item[(5)] For any $n\ge 2$, there exists $\delta_n>0$ such that for every $D\in\mathcal{D}$,
\[
D^n\cap{\rm DC1}_n^{\delta_n}(X,f)
\]
is a residual subset of $D^n$, i.e., contains a dense $G_\delta$-set of $D^n$.
\end{itemize}
\end{thm}

\begin{rem}
\normalfont
When $f\colon X\to X$ has the properties (1), (2), and (3), $f$ satisfies the shadowing property (see Lemma 2.3 in Section 2). Then, $h_{top}(f)=0$ implies that $X$ is a periodic orbit, or $(X,f)$ is topologically conjugate to an odometer (see \cite[Corollary 6]{Moo} and \cite[Theorem 4.4]{Ku}); therefore, except for these limited cases, $f$ satisfies the property (4).
\end{rem}

Following \cite{LLT}, we recall a simplified version of Mycielski's theorem (\cite[Theorem 1]{My}).  A topological space is said to be {\em perfect} if it has no isolated point. A subset $S$ of $X$ is said to be a {\em Mycielski set} if it is a union of countably many Cantor sets.

\begin{lem}
Let $X$ be a perfect compact metric space. If $R_n$ is a residual subset of $X^n$ for each $n\ge2$, then there is a Mycielski set $S$ which is dense in $X$ and satisfies
$(x_1,x_2,\dots,x_n)\in R_n$ for any $n\ge2$ and distinct $x_1,x_2,\dots,x_n\in S$.
\end{lem}

\begin{rem}
\normalfont
Note that every $D\in\mathcal{D}$ is perfect under the property (5). Then, the same property with Lemma 1.1 implies that for every $D\in\mathcal{D}$, there is a dense Mycielski subset $S$ of $D$ which is distributionally $n$-$\delta_n$-scrambled for all $n\ge2$ for some $\delta_n>0$.

Conversely, let $f\colon X\to X$ be a chain transitive map and $S$ be a distributionally $2$-scrambled subset of $X$. Since any $(x,y)\in{\rm DC1}_2(X,f)$ is a {\em proximal pair} for $f$, i.e.,
\[
\liminf_{k\to\infty}d(f^k(x),f^k(y))=0,
\]
we have $x\sim y$ (see Remark 1.2 (2)). This implies $S\subset D$ for some $D\in\mathcal{D}$. Thus, the property (5) ensures a maximal abundance of distributionally scrambled subsets under the chain transitivity.
\end{rem}

The main theorem in \cite{LLT} states that for a non-periodic transitive map $f\colon X\to X$ with the shadowing property, if $f$ has a periodic point, then there is a Mycielski subset $S$ of $X$ which is distributionally $n$-$\delta_n$-scrambled for all $n\ge2$ for some $\delta_n>0$, and moreover if $f$ has a fixed point, there is such $S$ also dense in $X$ (see \cite[Theorem 1.1]{LLT}).

Our Theorem 1.1 includes \cite[Theorem 1.1]{LLT} as explained in the next remark. Note that Theorem 1.1 above  does not exclude the case where $\mathcal{D}$ is an infinite set or $f$ has no periodic point.
 
\begin{rem}
\normalfont
Let $f\colon X\to X$ be a transitive map with the shadowing property. 
The transitivity is equivalent to the chain transitivity under the shadowing. If $f$ has a periodic point with a period $n>0$, then since $\sup_{\delta>0}|\mathcal{D}^\delta|\le n$, $\mathcal{D}=\mathcal{D}^\delta$ for sufficiently small $\delta>0$, so taking $A\in\mathcal{D}$, we have $|\mathcal{D}|\le n$ and
\[
X=\bigsqcup_{i=0}^{|\mathcal{D}|-1}f^i(A),
\]
where $f^i(A)$ is clopen in $X$ for every $0\le i\le|\mathcal{D}|-1$.

In this case, we easily see that $f$ has the properties (2) and (3). Moreover, under the additional assumption that $X$ is not a periodic orbit, (4) is also satisfied (see Remark 1.4). Thus, as shown in Remark 1.5, Theorem 1.1 ensures for every $D\in\mathcal{D}$ the existence of dense Mycielski subset $S$ of $D$ which is distributionally $n$-$\delta_n$-scrambled for all $n\ge2$ for some $\delta_n>0$. Note also that if $f$ has a fixed point, then $\mathcal{D}=\{X\}$.
\end{rem} 

Among various shadowing properties, {\em s-limit shadowing}, introduced in \cite{LS}, is a combination of the shadowing and limit shadowing properties. Its formal definition is given as follows.

\begin{defi}
\normalfont
Let $f\colon X\to X$ be a continuous map and let $\xi=(x_i)_{i\ge0}$ be a sequence of points in $X$.
\begin{itemize}
\item For $\delta>0$, $\xi$ is called a {\em $\delta$-limit-pseudo orbit} of $f$ if $d(f(x_i),x_{i+1})\le\delta$ for all $i\ge0$, and $\lim_{i\to\infty}d(f(x_i),x_{i+1})=0$.
\item For $\epsilon>0$, $\xi$ is said to be {\em $\epsilon$-limit shadowed} by $x\in X$ if $d(f^i(x),x_i)\leq \epsilon$ for all $i\ge 0$, and $\lim_{i\to\infty}d(f^i(x),x_i)=0$.
\item $f$ has the {\em s-limit shadowing property} if for any $\epsilon>0$, there is $\delta>0$ such that every $\delta$-limit-pseudo orbit of $f$ is $\epsilon$-limit shadowed by some point of $X$.
\end{itemize}
\end{defi}

Note that the s-limit shadowing has been proved to be $C^0$-dense in the space of continuous self-maps of compact topological manifolds \cite{MO}. In the next section, we prove the following implication.

\begin{prop}
Let $f\colon X\to X$ be a continuous map.  If $f$ is chain transitive and has the s-limit shadowing property, then $f$ satisfies the properties (2) and (3) in Theorem 1.1. 
\end{prop}

This proposition indicates a naturalness of the properties (2) and (3) in Theorem 1.1. By Theorem 1.1, we obtain the following corollary. 

\begin{cor}
Suppose that a continuous map $f\colon X\to X$ is chain transitive and has the s-limit shadowing property. If $h_{top}(f)>0$, then $f$ satisfies the property (5) in Theorem 1.1.
\end{cor}

\begin{rem}
\normalfont
In \cite{GOP}, an example is given of a continuous map $f\colon X\to X$ with  the following properties:
\begin{itemize}
\item[(1)] $f$ is transitive and has a fixed point,
\item[(2)] $f$ has the shadowing property,
\item[(3)] $f$ does not have the limit shadowing property.
\end{itemize}
The property $(1)$ implies that $f$ is chain transitive and satisfies $\mathcal{D}=\{X\}$. By this and (2), $f$ has the properties (2) and (3) in Theorem 1.1, obviously. However, since the s-limit shadowing implies the limit shadowing  (see \cite[Theorem 3.7]{BGO}), by (3), $f$ does not have the s-limit shadowing property. 
\end{rem}

We present one more corollary. Given a continuous map $f\colon X\to X$, an $x\in X$ is said to be a {\em chain recurrent point} for $f$ if for any $\delta>0$, there is a $\delta$-chain $(x_i)_{i=0}^k$ of $f$ with $x_0=x_k=x$. We say that $f$ is {\em chain recurrent} if every $x\in X $ is a chain recurrent point for $f$. When $f$ is chain recurrent, it is well-known that $X$ admits a partition into the {\em chain components} for $f$. 

Suppose that $f$ is chain recurrent and has the s-limit shadowing property. Let $\mathcal{C}(f)$ be the set of chain components for $f$. Then, we see that for any $C\in\mathcal{C}(f)$, $f|_C\colon C\to C$ is chain transitive and also has the s-limit shadowing property. Moreover, if $h_{top}(f)>0$, then there is an ergodic $f$-invariant Borel probability measure $\mu$ on $X$ such that $h_{\mu}(f)>0$. Note that $f|_{{\rm supp}(\mu)}\colon{\rm supp}(\mu)\to{\rm supp}(\mu)$ is transitive. By taking $C\in\mathcal{C}(f)$ with ${\rm supp}(\mu)\subset C$, we obtain $h_{top}(f|_C)>0$. Thus, by applying Corollary 1.1 to $f|_C$, we obtain the following (see also Remark 1.5).

\begin{cor}
Suppose that a continuous map $f\colon X\to X$ is chain recurrent and has the s-limit shadowing property. If $h_{top}(f)>0$, then there is a Mycielski subset $S$ of $X$ which is distributionally $n$-$\delta_n$-scrambled for all $n\ge2$ for some $\delta_n>0$.
\end{cor}

\section{Proof of the main results}

In this section, we prove Theorem 1.1 and also Proposition 1.1. Throughout this section, $f\colon X\to X$ is a chain transitive continuous map and $\mathcal{D}$ denotes the set of equivalence classes of $\sim$. For any $x\in X$, $\mathcal{D}(x)\in\mathcal{D}$ is defined by $x\in\mathcal{D}(x)$. Similarly, for every $\delta$-cyclic decomposition $\mathcal{D}^\delta$ of $X$, $\delta>0$, we define $\mathcal{D}^\delta(x)\in\mathcal{D}^\delta$ by $x\in\mathcal{D}^\delta(x)$ for all $x\in X$. Note that every $\mathcal{D}^\delta(x)$, $x\in X$, is a clopen subset of $X$. For any $A\subset X$ and $\epsilon>0$, put  $B_{\epsilon}(A)=\{x\in X\colon d(x,A)<\epsilon\}$, the $\epsilon$-neighborhood of $A$.

\begin{rem}
\normalfont
We have the following properties:
\begin{itemize}
\item[(1)] Given any $x\in X$, by the definition of $\sim$, $\mathcal{D}(x)=\bigcap_{\delta>0}\mathcal{D}^\delta(x)$, so for every $\epsilon>0$, there is $\delta>0$ such that $\mathcal{D}^\delta(x)\subset B _{\epsilon}(\mathcal{D}(x))$,
\item[(2)] Since $\sim$ is a closed relation in $X^2$, for any $x\in X$ and $\epsilon>0$, there is $\delta>0$ such that $d(x,y)<\delta$ implies $\mathcal{D}(y)\subset B_{\epsilon}(\mathcal{D}(x))$ for all $y\in X$. 
\end{itemize}
\end{rem}

The first lemma translates the property (3) in Theorem 1.1 into an equivalent property of $\mathcal{D}$.

\begin{lem}
The following properties are equivalent:
\begin{itemize}
\item[(1)] $\mathcal{D}(\cdot)\colon(X,d)\to(K(X),d_H)$ is continuous,
\item[(2)] For any $\epsilon>0$, there is $\delta>0$ such that $\mathcal{D}^\delta(x)\subset B_\epsilon(\mathcal{D}(x))$ for all $x\in X$.
\end{itemize}
\end{lem} 

\begin{proof}
The implication $(1)\Rightarrow(2)\colon$ Assume the contrary, then, there are $\epsilon>0$, a sequence $0<\delta_1>\delta_2>\cdots\to0$, and $x_k\in X$, $k\ge1$, such that for each $k\ge1$, we have $y_k\in\mathcal{D}^{\delta_k}(x_k)$ with $d(y_k,\mathcal{D}(x_k))\ge\epsilon$. We may assume $\lim_{k\to\infty}(x_k,y_k)=(x,y)$ for some $(x,y)\in X^2$. From the continuity of $\mathcal{D}(\cdot)$, it follows that $d(y,\mathcal{D}(x))\ge\epsilon$ and so $\mathcal{D}(x)\ne\mathcal{D}(y)$. Then, $\mathcal{D}^\delta(x)\ne\mathcal{D}^\delta(y)$ for some $\delta>0$. Taking $k\ge1$ with $(x_k,y_k)\in\mathcal{D}^\delta(x)\times\mathcal{D}^\delta(y)$ and $\delta_k<\delta$, we obtain
\[
\mathcal{D}^{\delta_k}(x_k)\cap\mathcal{D}^{\delta_k}(y_k)\subset\mathcal{D}^\delta(x_k)\cap\mathcal{D}^\delta(y_k)=\mathcal{D}^\delta(x)\cap\mathcal{D}^\delta(y)=\emptyset,
\]
because $\mathcal{D}^{\delta}\prec\mathcal{D}^{\delta_k}$ (see Remark 1.2 (1)). This implies $\mathcal{D}^{\delta_k}(x_k)\ne\mathcal{D}^{\delta_k}(y_k)$, that is, $y_k\notin\mathcal{D}^{\delta_k}(x_k)$, a contradiction; therefore, $(1)\Rightarrow(2)$ has been proved.
$ $\newline
$(2)\Rightarrow(1)\colon$ Fix $x\in X$ and for any $\epsilon>0$, take $\delta>0$ as in the property (2). Then, for every $y\in\mathcal{D}^\delta(x)$, we have
\[
\mathcal{D}(x)\subset\mathcal{D}^\delta(x)=\mathcal{D}^\delta(y)\subset B_\epsilon(\mathcal{D}(y))
\] 
and
\[
\mathcal{D}(y)\subset\mathcal{D}^\delta(y)=\mathcal{D}^\delta(x)\subset B_\epsilon(\mathcal{D}(x)),
\]
so $d_H(\mathcal{D}(x),\mathcal{D}(y))\le\epsilon$. Note that $\mathcal{D}^\delta(x)$ is a neighborhood of $x$ in $X$. Since $\epsilon>0$ is arbitrary, $\mathcal{D}(\cdot)$ is continuous at $x\in X$. Since $x\in X$ is also arbitrary, $\mathcal{D}(\cdot)$ is continuous.
\end{proof}

The second lemma concerns a uniform approximation of pseudo orbits by those along $\mathcal{D}$. 

\begin{lem}
Suppose that $f$ has the property (3) in Theorem 1.1. Then, for any $\gamma>0$, there exists $\delta>0$ such that for every $\delta$-pseudo orbit $(x_i)_{i\ge0}$ of $f$, there is a $(\mathcal{D},\gamma)$-pseudo orbit $(y_i)_{i\ge0}$ of $f$ with $x_0=y_0$ and $\sup_{i\ge0}d(x_i,y_i)<\gamma$.
\end{lem}  

\begin{proof}
Take $0<\beta<\gamma/3$ such that $d(a,b)<\beta$ implies $d(f(a),f(b))<\gamma/3$ for all $a,b\in X$. By Lemma 2.1, the property (3) in Theorem 1.1 is equivalent to the property (2) in Lemma 2.1. Then, it gives $0<\delta<\gamma/3$ such that $\mathcal{D}^\delta(x)\subset B_\beta(\mathcal{D}(x))$ for all $x\in X$.

Let $(x_i)_{i\ge0}$ be a $\delta$-pseudo orbit of $f$. By the property (P2) of $\sim_\delta$, we have $x_i\sim_\delta f^i(x_0)$ for all $i\ge0$ (see Remark 1.1). Put $y_0=x_0$. For each $i>0$, since $x_i\in\mathcal{D}^\delta(f^i(x_0))\subset B_\beta(\mathcal{D}(f^i(x_0)))$, there is $y_i\in\mathcal{D}(f^i(x_0))$ such that $d(x_i,y_i)<\beta$. Fix such $(y_i)_{i\ge0}$. Given any $i\ge0$, since $\sim$ is $(f\times f)$-invariant, $f(y_i)\sim f^{i+1}(x_0)$. Combining this with $y_{i+1}\sim f^{i+1}(x_0)$, we obtain $f(y_i)\sim y_{i+1}$. Also, for each $i\ge0$, it holds that
\begin{align*}
d(f(y_i),y_{i+1})&\le d(f(y_i),f(x_i))+d(f(x_i),x_{i+1})+d(x_{i+1},y_{i+1})\\
&<\gamma/3+\delta+\beta<\gamma.
\end{align*}
Since $\sup_{i\ge0}d(x_i,y_i)\le\beta<\gamma$, $(y_i)_{i\ge0}$ gives the desired $(\mathcal{D},\gamma)$-pseudo orbit of $f$. 
\end{proof}

The following lemma is mentioned in Section 1.

\begin{lem}
If $f$ has the properties (2) and (3) in Theorem 1.1, then $f$ has the shadowing property.
\end{lem}

\begin{proof}
Given any $\epsilon>0$, take $0<\gamma<\epsilon/2$ such that every $(\mathcal{D},\gamma)$-pseudo orbit of $f$ is $\epsilon/2$-shadowed by some point of $X$. For this $\gamma$, take $\delta>0$ as in Lemma 2.2.

Then, for every $\delta$-pseudo orbit $\xi=(x_i)_{i\ge0}$ of $f$, we have a $(\mathcal{D},\gamma)$-pseudo orbit $(y_i)_{i\ge0}$ of $f$ with $\sup_{i\ge0}d(x_i,y_i)<\gamma$, which is $\epsilon/2$-shadowed by some $z\in X$. It follows that
\[
d(f^i(z),x_i)\le d(f^i(z),y_i)+d(y_i,x_i)<\epsilon/2+\gamma<\epsilon
\]
for all $i\ge0$, that is, $\xi$ is $\epsilon$-shadowed by $z$. Since $\epsilon>0$ is arbitrary, $f$ has the shadowing property.
\end{proof}

For $n\ge 2$ and $r>0$, we say that an $n$-tuple $(x_1,x_2,\dots,x_n)\in X^n$ is {\em $r$-distal} if
\[
\inf_{k\ge0}\min_{1\le i<j\le n}d(f^k(x_i),f^k(x_j))\ge r.
\]
While the next lemma can be proved by similar arguments as in the proof of \cite[Lemma 3.3]{LLT} and also \cite[Theorem 1.2]{Ka}, we outline its proof. Note that under the assumption that $f$ is chain transitive, the relation $\sim$ defined in \cite{Ka} coincides with the relation $\sim$ in this paper.
 
\begin{lem}
If $f$ has the shadowing property and satisfies $h_{top}(f)>0$, then for any $n\ge2$, there are $r_n>0$ and an $r_n$-distal $n$-tuple $(x_1,x_2,\dots,x_n)\in X^n$ such that $\{x_1,x_2,\dots,x_n\}\subset D_n$ for some $D_n\in\mathcal{D}$.
\end{lem}

\begin{proof}
Since $h_{top}(f)>0$, there is an entropy pair $(z,w)\in X^2$ for $f$. By Lemma 3.3 in \cite{Ka}, we have $z\ne w$ and $z\sim w$. Since $f$ has the shadowing property, this implies the existence of $a>0$ and a closed $f^a$-invariant subset $Y$ of $X$ such that there is a factor map
\[
\pi\colon (Y,f^a)\to(\Sigma,\sigma),
\]
where $\Sigma=\{0,1\}^\mathbb{N}$, and $\sigma\colon\Sigma\to\Sigma$ is the shift map. Fix $n\ge2$ and take distinct periodic points $y_1,y_2,\dots,y_n\in\Sigma$ for $\sigma$.  Since $\sigma$ has upe of order $n$, $(y_1,y_2,\dots,y_n)\in\Sigma^n$ is an entropy $n$-tuple for $\sigma$. Note that it is also $s_n$-distal for some $s_n>0$. Since $\pi$ is a factor map, there is
\[
(x_1,x_2,\dots,x_n)\in\pi^{-1}(y_1)\times\pi^{-1}(y_2)\times\cdots\times\pi^{-1}(y_n)
\]
which is an entropy $n$-tuple for $f^a\colon Y\to Y$ so for $f$. Then, $(x_i,x_j)$ is an entropy pair for $f$ for all $1\le i<j\le n$. Again by \cite[Lemma 3.3]{Ka}, we obtain $x_i\sim x_j$ for all $1\le i<j\le n$, implying $\{x_1,x_2,\dots,x_n\}\subset D_n$ for some $D_n\in\mathcal{D}$. Also, $(x_1,x_2,\dots,x_n)\in X^n$ is obviously $r_n$-distal for some $r_n>0$; therefore, the lemma has been proved.
\end{proof}

The next lemma implies that distal $n$-tuples, $n\ge2$, whose existence has been proved in Lemma 2.4, are distributed to all components of $\mathcal{D}$.

\begin{lem}
Let $n\ge 2$. Suppose that there are $r_n>0$ and an $r_n$-distal $n$-tuple $(x_1,x_2,\dots,x_n)\in X^n$ such that $\{x_1,x_2,\dots,x_n\}\subset D_n$ for some $D_n\in\mathcal{D}$. Then, for every $D\in\mathcal{D}$, there is an $r_n$-distal $n$-tuple $(y_1,y_2,\dots,y_n)\in X^n$ such that $\{y_1,y_2,\dots,y_n\}\subset D$.
\end{lem}

\begin{proof}
In \cite{RW}, the factor of $(X,f)$ with respect to $\sim$ is shown to be a periodic orbit or topologically conjugate to an odometer and so minimal in both cases (see the proof of \cite[Theorem 6]{RW}). Note that $\mathcal{D}=X\slash\sim$, and put $\pi=\mathcal{D}(\cdot)\colon X\to\mathcal{D}$, the quotient map. Then, it gives the factor map $\pi\colon(X,f)\to(\mathcal{D},g)$, where $g\colon\mathcal{D}\to\mathcal{D}$ is a map defined by $g\circ\pi=\pi\circ f$.

Given any $D\in\mathcal{D}$, since $g$ is minimal, there is a sequence $0\le n_1<n_2<\cdots$ such that $\lim_{k\to\infty}g^{n_k}(D_n)=D$. We may assume $\lim_{k\to\infty}f^{n_k}(x_i)=y_i$ for all $1\le i\le n$ for some $y_i\in X$. Then, $(y_1,y_2,\dots,y_n)\in X^n$ is $r_n$-distal. Also, we have
\begin{align*}
\pi(y_i)&=\lim_{k\to\infty}\pi(f^{n_k}(x_i))=\lim_{k\to\infty}g^{n_k}(\pi(x_i))=\lim_{k\to\infty}g^{n_k}(D_n)=D
\end{align*}
for all $1\le i\le n$, i.e., $\{y_1,y_2,\dots,y_n\}\subset D$. Thus, the lemma has been proved.
\end{proof}

The last lemma states that, for $D\in\mathcal{D}$, each $y\in D$ is approximated by $z\in D$ whose orbit is asymptotically close to the orbit of any reference point $x\in D$. 

\begin{lem}
Suppose that $f$ has the properties (2) and (3) in Theorem 1.1. Let $D\in\mathcal{D}$ and $x\in D$. Then, for any $y\in D$ and $\epsilon>0$, there is $z\in D$ such that $d(y,z)<\epsilon$ and $\limsup_{k\to\infty}d(f^k(x),f^k(z))<\epsilon$. 
\end{lem}

\begin{proof}
Take $0<\gamma<\epsilon/2$ such that every $(\mathcal{D,\gamma})$-pseudo orbit $(x_i)_{i\ge0}$ of $f$ is $\epsilon/2$-shadowed by some $z\in X$ with $x_0\sim z$. For this $\gamma$, take $\delta>0$ as in Lemma 2.2. Also, choose $N>0$ as in the property (P4) of $\sim_\delta$ (see Remark 1.1).

Note that $y \in D$ implies $x\sim y$ and so $x\sim_\delta y$. Since $x\sim_\delta f^{mN}(x)$, we have $y\sim_\delta f^{mN}(x)$. Then, the choice of $N$ gives a $\delta$-chain $\alpha=(y_i)_{i=0}^{mN}$ of $f$ with $y_0=y$ and $y_{mN}=f^{mN}(x)$. Let
\[
\beta=(f^{mN}(x),f^{mN+1}(x),\dots)
\]
and $\xi=\alpha\beta=(z_i)_{i\ge0}$. Since $\xi$ is a $\delta$-pseudo orbit of $f$, there is a $(\mathcal{D},\gamma)$-pseudo orbit $(x_i)_{i\ge0}$ of $f$ with $z_0=x_0$ and $\sup_{i\ge0}d(z_i,x_i)<\gamma$, which is $\epsilon/2$-shadowed by some $z\in X$ with $x_0\sim z$. Because $y=y_0=z_0=x_0$, we have $y\sim z$ and so $z\in D$. Note that $d(y,z)=d(x_0,z)\le\epsilon/2<\epsilon$. Also, we have
\[
d(f^i(x),f^i(z))=d(z_i,f^i(z))\le d(z_i,x_i)+d(x_i,f^i(z))<\gamma+\epsilon/2
\]
for every $i\ge mN$, so $\limsup_{k\to\infty}d(f^k(x),f^k(z))\le\gamma+\epsilon/2<\epsilon$. This completes the proof.
\end{proof}

By the above lemmas, we prove Theorem 1.1.

\begin{proof}[Proof of Theorem 1.1]
For any $n\ge2$, from Lemmas 2.3, 2.4, and 2.5, it follows that there is $r_n>0$ such that for every $D\in\mathcal{D}$, there is an $r_n$-distal $n$-tuple $(y_1,y_2,\dots,y_n)\in X^n$ such that $\{y_1,y_2,\dots,y_n\}\subset D$. Fix $0<\delta_n<r_n$. Given any $D\in\mathcal{D}$, we show that $D^n\cap{\rm DC1}_n^{\delta_n}(X,f)$ is a residual subset of $D^n$.

Following \cite{LLT}, put
\[
A_n^\sigma(D,f)=\{(x_1,x_2,\dots,x_n)\in D^n\colon\limsup_{k\to\infty}\max_{1\le i<j\le n}d(f^k(x_i),f^k(x_j))<\sigma\}
\]
for $\sigma>0$, and
\[
D_n^{\delta_n}(D,f)=\{(x_1,x_2,\dots,x_n)\in D^n\colon\liminf_{k\to\infty}\min_{1\le i<j\le n}d(f^k(x_i),f^k(x_j))>\delta_n\}.
\]
Then, by Lemma 2.6, we see that $A_n^\sigma(D,f)$ is dense in $D^n$ for any $\sigma>0$, and also $D_n^{\delta_n}(D,f)$ is dense in $D^n$. Again following \cite{LLT}, put
\begin{align*}
R(l,q,m)=&\{(x_1,x_2,\dots,x_n)\in D^n\colon \\
&\frac{1}{m}|\{0\le k\le m-1\colon\max_{1\le i<j \le n}d(f^k(x_i),f^k(x_j))<\frac{1}{l}\}|>1-\frac{1}{q}\}
\end{align*}
for $l,q,m\ge1$, and
\begin{align*}
S^{\delta_n}(q,m)=&\{(x_1,x_2,\dots,x_n)\in D^n\colon \\
&\frac{1}{m}|\{0\le k\le m-1\colon\min_{1\le i<j\le n}d(f^k(x_i),f^k(x_j))>\delta_n\}|>1-\frac{1}{q}\}
\end{align*}
for $q,m\ge1$. Note that those are open subsets of $D^n$. Letting
\[
R=\cap_{l=1}^\infty\cap_{q=1}^\infty\cap_{p=1}^\infty\cup_{m=p}^\infty R(l,q,m)
\quad
\text{and}
\quad
S^{\delta_n}=\cap_{q=1}^\infty\cap_{p=1}^\infty\cup_{m=p}^\infty S^{\delta_n}(q,m),
\]
we obtain $R\cap S^{\delta_n}\subset D^n\cap{\rm DC1}_n^{\delta_n}(X,f)$. Since
\[
A_n^{\frac{1}{l}}(D,f)\subset\cup_{m=p}^\infty R(l,q,m)
\quad
\text{and}
\quad
D_n^{\delta_n}(D,f)\subset\cup_{m=p}^\infty S^{\delta_n}(q,m)
\]
for any $l,q,p\ge1$, $R$, $S^{\delta_n}$, and so $R\cap S^{\delta_n}
$ are dense $G_\delta$-sets of $D^n$. This implies $D^n\cap{\rm DC1}_n^{\delta_n}(X,f)$ is a residual subset of $D^n$, completing the proof.
\end{proof}

Finally, we give a proof of Proposition 1.1.

\begin{proof}[Proof of Proposition 1.1]
First, we show that $f$ has the  property (2). For any $\epsilon>0$, take $\delta>0$ such that every $\delta$-limit-pseudo orbit of $f$  is $\epsilon$-limit shadowed by some point of $X$. Given any $(\mathcal{D},\delta)$-pseudo orbit $\xi=(x_i)_{i\ge0}$ of $f$, we have $x_i\sim f^i(x_0)$ for all $i\ge0$. For each $n\ge0$, consider
\[
\xi_n=(x_0,x_1,\dots,x_n,f(x_n),f^2(x_n),\dots),
\]
a $\delta$-limit-pseudo orbit of $f$, which is $\epsilon$-limit shadowed by $z_n\in X$. Assume $x_0\not\sim z_n$ for some $n\ge0$. Then, there is $\gamma>0$ such that $x_0\not\sim_\gamma z_n$, implying $f^n(x_0)\not\sim_\gamma f^n(z_n)$. Since $x_n\sim f^n(x_0)$, we have $x_n\sim_\gamma f^n(x_0)$, so $x_n\not\sim_\gamma f^n(z_n)$. However, this contradicts that
\[
\lim_{i\to\infty}d(f^i(x_n),f^{i+n}(z_n))=0.
\]
Thus, we obtain $x_0\sim z_n$ for every $n\ge0$. Note that $d(x_i,f^i(z_n))\le\epsilon$ for any $0\le i\le n$ and $n\ge0$. Taking a sequence $0\le n_1<n_2<\cdots$ and $z\in X$ with
\[
\lim_{k\to\infty}z_{n_k}=z,
\]
we obtain $x_0\sim z$ (since $\sim$ is closed) and $d(x_i,f^i(z))\le\epsilon$ for all $i\ge0$. Since $\xi$ and then $\epsilon>0$ are arbitrary, $f$ has the shadowing property along $\mathcal{D}$.

Next, we show that $f$ has the property (3). Assume the contrary, that is, $\mathcal{D}(\cdot)$ is not continuous at some $x\in X$, to derive a contradiction. Then, by Remark 2.1 (2), we see that there are $\epsilon>0$ and $y_k\in X$, $k\ge1$, such that $\lim_{k\to\infty}y_k=x$ and $\mathcal{D}(x)\not\subset B_{\epsilon}(\mathcal{D}(y_k))$ for every $k\ge1$. For this $\epsilon$, take $\delta>0$ such that every $\delta$-limit-pseudo orbit of $f$  is $\epsilon/2$-limit shadowed by some point of $X$. Choose $N>0$ as in the property (P4) of $\sim_\delta$ (see Remark 1.1). Fix $k\ge1$ with $y_k\in\mathcal{D}^\delta(x)$ and take $z\in\mathcal{D}(x)$ such that $z\notin B_{\epsilon}(\mathcal{D}(y_k))$. Since $\{y_k,z\}\subset\mathcal{D}^\delta(x)$, we have $y_k\sim_\delta z$. This with $y_k\sim_\delta f^{mN}(y_k)$ implies $z\sim_\delta f^{mN}(y_k)$. The choice of $N$ gives a $\delta$-chain $\alpha=(x_i)_{i=0}^{mN}$ of $f$ with $x_0=z$ and $x_{mN}=f^{mN}(y_k)$. Let
\[
\beta=(f^{mN}(y_k),f^{mN+1}(y_k),\dots)
\]
and $\xi=\alpha\beta=(z_i)_{i\ge0}$, a $\delta$-limit-pseudo orbit of $f$, $\epsilon/2$-limit shadowed by $w\in X$. Since
\[
d(f^i(y_k),f^i(w))=d(z_i,f^i(w))
\]
for every $i\ge mN$,
\[
\lim_{i\to\infty}d(f^i(y_k),f^i(w))=0.
\]
This yields $y_k\sim w$ and so $w\in\mathcal{D}(y_k)$  (see Remark 1.2 (2)). By $d(z,w)=d(x_0,w)=d(z_0,w)\le\epsilon/2$, we obtain $z\in B_{\epsilon}(\mathcal{D}(y_k))$, a contradiction. Thus, $\mathcal{D}(\cdot)$ is continuous, and the proof has been completed.
\end{proof}

\section*{Acknowledgements}

This work was supported by JSPS KAKENHI Grant Number JP20J01143.


\begin{thebibliography}{99}

\bibitem{AC} T. Arai, N. Chinen, P-chaos implies distributional chaos and chaos in
the sense of Devaney with positive topological entropy. Topology Appl. 154 (2007), 1254--1262.

\bibitem{BGO} A.D. Barwell, C. Good, P. Oprocha, Shadowing and expansivity in subspaces. Fund. Math. 219 (2012), 223--243.

\bibitem{GOP} C. Good, P. Oprocha, M. Puljiz, Shadowing, asymptotic shadowing and s-limit shadowing. Fund. Math. 244 (2019), 287--312.

\bibitem{Ka} N. Kawaguchi, Distributionally chaotic maps are $C^0$-dense. Proc. Amer. Math. Soc. 147 (2019), 5339--5348.

\bibitem{Ku}P. K\r{u}rka, Topological and Symbolic Dynamics. Societe Mathematique de France, Paris, 2003.

\bibitem{LS} K. Lee, K. Sakai, Various shadowing properties and their equivalence. Discrete Contin. Dyn. Syst. 13 (2005), 533--540.

\bibitem{LLT} J. Li, J. Li, S. Tu, Devaney chaos plus shadowing implies distributional chaos. Chaos 26 (2016), 093103, 6 pp.

\bibitem{LO} J. Li, P. Oprocha, On $n$-scrambled tuples and distributional chaos in a sequence. J. Differ. Equations Appl. 19 (2013), 927--941.

\bibitem{MO} M. Mazur, P. Oprocha, S-limit shadowing is $C^0$-dense. J. Math. Anal. Appl. 408 (2013), 465--475.

\bibitem{Moo} T.K.S. Moothathu, Implications of pseudo-orbit tracing property for continuous maps on compacta. Topology Appl. 158 (2011), 2232--2239.

\bibitem{My} J. Mycielski, Independent sets in topological algebras. Fund. Math. 55 (1964), 139--147.

\bibitem{OW} P. Oprocha, X. Wu, On averaged tracing of periodic average pseudo orbits. Discrete Contin. Dyn. Syst. 37 (2017), 4943--4957.

\bibitem{RW} D. Richeson, J. Wiseman, Chain recurrence rates and topological entropy. Topology Appl. 156 (2008), 251--261.

\bibitem{TF} F. Tan, H. Fu, On distributional $n$-chaos. Acta Math. Sci. (Engl. Ed.) 34 (2014), 1473--1480.

\end{thebibliography}
\end{document}